\documentclass[a4paper, 11 pt, twoside]{amsart}
\usepackage{srollens-en}[2011/03/03]
\usepackage[left = 3.5cm, right = 3.5cm, headsep = 6mm,
footskip = 10mm, top = 35mm, bottom = 35mm, footnotesep=5mm, headheight =
2cm]{geometry}

\newcommand{\tsum}{{\textstyle\sum}}

\usepackage[unicode,bookmarks, pdftex]{hyperref}
\hypersetup{colorlinks=true,citecolor=NavyBlue,linkcolor=NavyBlue,urlcolor=Orange, pdfpagemode=UseNone, breaklinks=true}

\title{Special triple covers of algebraic surfaces}

\author{Nicolina Istrati}
\address{Nicolina Istrati\\FB 12/Mathematik und Informatik\\
Philipps-Universit\"at Marburg\\
Hans-Meer\-wein-Str. 6\\
35032 Marburg\\
Germany}
\email{nicolina.istrati@uni-marburg.de}

\author{Piotr Pokora}
\address{Department of Mathematics,
Pedagogical University of Krakow,
Podchor\c a\.zych 2,
PL-30-084 Krak\'ow, Poland}
\email{piotr.pokora@up.krakow.pl, piotr.pokora@up.krakow.pl}

\author{S\"onke Rollenske}
\address{S\"onke Rollenske\\FB 12/Mathematik und Informatik\\
Philipps-Universit\"at Marburg\\
Hans-Meerwein-Str. 6\\
35032 Marburg\\
Germany}
\email{rollenske@mathematik.uni-marburg.de}

\begin{document}
\begin{abstract}
We  study \emph{special triple covers} $f\colon T \to S$ of algebraic surfaces, where the Tschirnhausen bundle $\ke = \left(f_*\ko_T/\ko_S\right)^\vee$ is a quotient of a split rank three vector bundle, and we provide several necessary and sufficient criteria for the existence.

As an application, we give a complete classification of special triple planes, finding among others two nice families of K3 surfaces.

\end{abstract}
\subjclass[2010]{14J10, 14J29}
\keywords{triple covers, K3 surfaces, surface of general type}

\maketitle
\setcounter{tocdepth}{1}
\tableofcontents

\section{Introduction}
Every algebraic geometer needs his toolbox of methods to construct new algebraic varieties. One of the most prominent of those is taking finite branched covers  $f\colon T \to S$, possibly followed by desingularisation. Here we mostly restrict to the case where both $S$ and $T$ are smooth. Then the finite map $f$  is determined by the $\ko_S$-algebra structure on $f_*\ko_T$, so  a major role is played by the Tschirnhausen bundle $\ke = \left(f_*\ko_T/\ko_S\right)^\vee$. 

The easiest case  is when $f$ is an  abelian Galois cover, for then the symmetries force $\ke$ to split into line bundles and the multiplication maps can be specifed in terms of divisors, see e.g. \cite{pardini91, catanese99}. This includes all maps of degree $2$. The non-abelian case, which is much more challenging, has also attracted some interest in recent years in different contexts, as for example in \cite{MR3660905,MR4334321,MR2866071,MR3443036, Tokunga}.

General finite maps have been studied in low degrees by Miranda \cite{Miranda, Hahn-MIranda1999} and by Casnati and Ekedahl in a series of papers  starting with  \cite{CE96}. Roughly speaking, the theory becomes more complicated in two ways: the Tschirnhausen bundle is no longer split and the description of the multiplication map is less straightforward.

For covers of  degree three the bundle $\ke$ is of rank two and the multiplication map is encoded by a section  $\eta \in H^0(S, \Sym^3 \ke \tensor \det \ke ^\vee)$ -- see Section \ref{sect: triple covers}. This situation is reasonably approachable if one can control the Tschirnhausen bundle. In general, this can be quite difficult but has been successfully carried out in special cases, e.g. \cite{MR3885155, MR1768173, MR3128997}. In the context of triple covers, it is also worth recalling some recent work regarding the classification of trigonal surfaces which uses generically finite triple covers over some ruled surface, not necessarily smooth, via cubic equations, as  explained in  \cite{MR1926442}.

In this paper we explore \emph{special triple covers}, corresponding to the case when the Tschirn\-hausen bundle is what we call nearly split, i.e., it is a quotient of a direct sum of three line bundles. This notion, introduced and characterised in Section \ref{sect: VB}, includes the split case but allows for extra flexibility in the bundle and extra challenges in the construction of the defining section $\eta$.

After recalling the general theory of triple covers and the case of split Tschirn\-hausen bundle in Section \ref{sect: triple covers}, we specialise to special triple covers in Section \ref{sect: special triple covers} and deduce their invariants in terms of the building data of the bundle $\ke$.  

In Section \ref{sect: suffcient and necessary conditions}, we derive several necessary and sufficient criteria for the existence of triple covers with a given nearly split $\ke$.  When the building data of the Tschirnhausen bundle is commensurable, we obtain a rather complete picture regarding admissibility in Theorem \ref{thm: gci existence}. An instance of this result for the case $S=\IP^2$ reads as follows.
\begin{custom}[Theorem \ref{thm: admissible triple planes}]
Let $\ke$ be a rank two vector bundle on $\IP^2$ sitting in an exact sequence:
\begin{equation}\label{eq: e}
 \begin{tikzcd}
  0 \rar & \ko(d)\rar{(\gamma_1,\gamma_2,\gamma_3)} & \bigoplus_{i=1}^3\ko(d+c_i)\rar & \ke\rar& 0
 \end{tikzcd} \end{equation}
 with $c_1\geq c_2\geq c_3\geq 1$. Then $\ke$ is the Tschirnhausen bundle of a smooth triple cover $T\rightarrow \IP^2$ if and only if one of the following conditions holds:
 \begin{enumerate}
 \item $d+c_3-c_1\geq 0$
 \item $d+c_3-c_2\geq 0$, $d-c_1+2c_2-c_3\geq 0$ and the curve $Z(\gamma_3)$ is smooth at the points $Z(\gamma_2,\gamma_3)$.
 \end{enumerate}
\end{custom}
 
As a first application, we complement  recent work of Garbagnati and Penegini \cite{GP21}, by giving in Example \ref{exam: covers of K3} several simple but non-split triple covers of K3 surfaces, both properly elliptic and of general type. 

If the Tschirnhausen bundle is sufficiently positive, then the corresponding triple cover should be minimal of general type.  We show how this can be controlled and where such  surfaces appear in the geography of surfaces of general type in Section~\ref{sect: aymptotics}. However, we also point out that a certain balance has to be maintained in the bundle, otherwise it can not be realised as a Tschirnhausen bundle. 

To showcase the applicability of our results, we treat special triple planes $f\colon T \to \IP^2$ exhaustively in Section \ref{sect: triple planes}.
\begin{table}[t]
  \caption{Special triple planes not minimal of general type}
    \label{tab: special triple planes not of gen type}
\centering
\begin{tabular}{cccc|ccc}
 \toprule
 $\kappa(T)$ &$p_g(T)$ & $q(T)$ & $K_T^2$  & $d$ & $( c_1, c_2, c_3)$ \\
 \midrule
$-\infty$ & $0$ & $0$ & $8$ &  $0$ & $(1,1,0)$ \\
$-\infty$ &  $0$ & $0$ & $3$ & $0$ & $(2,1,0)$ \\
$-\infty$ & $0$ & $0$ & $-1$ &  $0$ & $(2,2,0)$\\
$-\infty$ & $0$ & $1$ &  $0$ &  $0$ & $(1,1,1)$ \\
$-\infty$ & $0$ & $1$ &  $-9$&  $0$ & $(2,2,2)$ \\
$-\infty$ & $0$ & $0$ &  $-4$ & $1$ & $(1,1,1)$\\
\midrule
$0$ &  $1$ & $0$ & $-1$ & $0$ & $(3,2,0)$ \\
$0$ & $1$ & $0$ &  $-3$ &  $1$ & $(2,1,1)$ \\
 \midrule
$1$ & $2$ & $0$ & $0$ &  $0$ & $(3,3,0)$ \\
$1$ & $2$ & $0$ &  $-1$ &  $1$ & $(2,2,1)$  \\
$1$ & $3$ & $1$ &  $0$ &  $0$ & $(3,3,3)$ \\
 \midrule 
$2$  &  $4$ & $0$ &  $5$ &  $1$ & $(3,2,1)$ \\
$2$ & $7$ & $0$ &  $15$ &  $1$ & $(4,2,1)$ \\
    \bottomrule
 \end{tabular}
  \end{table}

\begin{custom}[Theorem \ref{prop: triple planes not general type} and Proposition \ref{prop: triple planes non min of gen type}]
Let $f\colon T\rightarrow \IP^2$ be a special triple cover, so that its Tschirnhausen bundle $\ke$ sits in an exact sequence \eqref{eq: e} with $c_1\geq c_2\geq c_3\geq 0$. It $T$ is not minimal of general type, then it is in one of the families given in Table \ref{tab: special triple planes not of gen type}.
\end{custom} 

In the rest of the section, we do a detailed study of each family appearing in the above table. Especially the two families of K3 surfaces admit a nice alternative description, see Example~\ref{exam: quartic} and Section~\ref{sect: K3 family }. 

\newpage
\subsection*{Acknowlegements}
The suggestion to study special triple covers arose in Fynn Brunzel's Master thesis \cite{Fynn}\footnote{In the notation of this paper, he studied  special triple planes with $(c_1, c_2, c_3) = (c_1, 1,1)$, where the ideal sheaf in \eqref{eq: Serre seq for E} is the ideal sheaf of a point.}.
Significant simplifications where found when the third author was visiting Krakow. We would also like to thank \L ucja Farnik for discussions.

The second author was partially supported by the National Science Center (Poland) Sonata Grant Nr \textbf{2018/31/D/ST1/00177}.

\subsection*{Notation and conventions}
We work over a fixed smooth complex projective or just compact complex surface $S$. A canonical divisor is denoted by $K_S$, the canonical bundle with $\omega_S$. 
With $\chi(S) = \chi(\ko_S)$ we denote the holomorphic Euler characteristic.

\section{Nearly split rank two bundles}\label{sect: VB}

Recall that a vector bundle is called split if it is a direct sum of line bundles. We now want to consider, for rank two vector bundles,  the next best case, where we can find a surjection from three line bundles.

\begin{prop}\label{prop: nearly split E}
 Let $\ke$ be a vector bundle of rank $2$ on $S$. Assume that there exist three line bundles $\kl_1, \kl_2, \kl_3$ and  a surjection 
 $\pi \colon \kl_1\oplus \kl_2\oplus \kl_3 \onto \ke$.
 
Then there exists a line bundle $\kl$ and three curves $C_1, C_2, C_3$ such that 
\begin{enumerate}
 \item no two curves have a common component,
 \item the triple intersection $C_1\cap C_2\cap C_3 = \emptyset$,
 \item there is an exact sequence
 \begin{equation}\label{eq: split resolution of E}
 \begin{tikzcd}
  0 \rar & \kl\rar{\alpha} & \bigoplus\kl(C_i)\rar & \rar\ke& 0,
 \end{tikzcd}
\end{equation}
where $\alpha$ is given by multiplication with sections defining the curves $C_i$.
\end{enumerate}
Conversely, given a line bundle and three curves satisfying $(i)$ and $(ii)$, there exists a bundle $\ke$ as in $(iii)$.

We call $\ke$ the \emph{nearly split vector bundle} associated to $(\kl, C_1, C_2, C_3)$.
\end{prop}

\begin{rem}\label{rem: split}
It may well be that a nearly split vector bundle in our sense also admits a splitting into line bundles. For example, if in the construction the curve $C_3$ is empty, then the sequence \eqref{eq: split resolution of E} is split exact. In other words, split bundles are also nearly  split. 
\end{rem}

\begin{proof}
First note that given a sequence \eqref{eq: split resolution of E} the quotient $\ke$ is a vector bundle if and only if the map has everywhere full rank if and only if $(ii)$ is satisfied. 

Now assume the surjection $\pi$ is given and let $\kl = \ker \pi$, which is a line bundle, because it is reflexive of rank $1$,  and let  $C_i$ be the divisor defining the map $\kl\to \kl_i$.  
 
 Then there is an exact sequence \eqref{eq: split resolution of E}, and the triple intersection is zero because $\ke$ is a vector bundle. 

 Now assume that $C_2 = C_2'+D$ and $C_3 =C_3'+D$ such that $C'_2$ and $C'_3$ have no common component. Then if we factor out $\kl_1 = \kl(C_1)$ we get a diagram
 \begin{equation}\label{diag: resolution / Serre seq / koszul res}
 \begin{tikzcd}[ampersand replacement  = \&]
 {} \&  \&   0\dar \&  0\dar \\    
 \& \&  \kl
  \dar
 \rar[equal] \& \kl\dar{\mat{\delta\gamma_2 & \delta\gamma_3}^T} \\
  0\rar \&  \kl(C_1) \rar\dar[equal] \&  \kl(C_1) \oplus \kl (C_2) \oplus \kl(C_3) \dar\rar \& \kl (C_2) \oplus \kl(C_3) \rar\dar{\mat{-\gamma_3& \gamma_2}}  \&  0 \\
0 \rar  \&  \kl(C_1 ) \rar \&  \ke\rar\dar \& \kq \rar\dar\&  0  \\
\& \&  0 \&  0 \\
    \end{tikzcd},
     \end{equation}
where $\delta$ is an equation for $D$. 
The first map in the last column can be factored through $\kl(D)$, so we get
\begin{small}
 \[ 
 \begin{tikzcd}
  0 \rar & \kl \rar \dar &\kl(C_2) \oplus \kl(C_3)\rar \dar[equal] & \kq \rar\dar & 0\\
  0 \rar & \kl(D) \rar &\kl(C_2'+D) \oplus \kl(C_3'+D) \rar & \kl(D+C_2'+C_3')\tensor \ki_{C'_2\cap C'_3}  \rar & 0
 \end{tikzcd},
\]
\end{small}
where the lower row is the twisted Koszul-complex of the complete intersection $C'_2\cap C'_3$. By the snake-Lemma the torsion submodule of $\kq$ is $\kl(D)|_D$, so 
the homomorphism $\kl(C_1) \to \ke$ vanishes along $D$ and thus factors over $\kl(C_1+D)$. 

Setting $\kl'= \kl(D)$, we obtain 
\[
\begin{tikzcd}
  0 \rar & \kl'\rar & \kl'(C_1)\oplus \kl'(C_2')\oplus\kl'(C_3')\rar & \rar\ke& 0,
 \end{tikzcd}
\]
where we have removed the common components from $C_2$ and $C_3$. Repeating this for the other curves, we get to a situation where $(i)$ holds as well. 
\end{proof}

\begin{rem}
Given a sequence \eqref{eq: split resolution of E} satisfying the conditions of Proposition \ref{prop: nearly split E}, the last column of \eqref{diag: resolution / Serre seq / koszul res} becomes the (twisted) Koszul resolution for the ideal $\ki_Z$  of the complete intersection subscheme  $Z = C_2\cap C_3$ (compare Proposition \ref{prop: resolution ideal}). 
Thus we have an exact sequence
\begin{equation}\label{eq: Serre seq for E}
 0 \to \kl(C_1) \to \ke \to \kl(C_2+C_3) \ki_Z\to 0, 
\end{equation}
where we can clearly permute the roles of the $C_i$'s.
This is the usual sequence we have at our disposal when one studies rank two vector bundles, compare \cite[Ch. IV.11]{BHPV}

Conversely, if a rank two bundle $\ke$ fits into an exact sequence as in \eqref{eq: Serre seq for E} with $Z= C_2 \cap C_3$, then it is nearly split if and only if the extension class maps to zero under the map
\begin{small}
 \[ 
 \begin{tikzcd}
\Ext^1(\kl(C_2+C_3)\ki_Z, \kl(C_1)) \rar & \Ext^1(\kl(C_2)\oplus \kl(C_3), \kl(C_1))\dar[equal] \\
&H^1(\ko(C_1-C_2))\oplus H^1(\ko(C_1-C_3)) \end{tikzcd}.\]
\end{small}
It is not hard to construct examples where this is not the case.
\end{rem}

\begin{prop}\label{prop: Cherncl}
 Let $\ke$ be a nearly split vector bundle associated to $(\kl, C_1, C_2, C_3)$. Then
 \begin{gather*}
 \det\ke = \kl^2(C_1+C_2+C_3)\\
  c_1(\ke) = 2c_1(\kl) +\sum_i [C_i]\\
  c_2(\ke) =\kl^2+\kl.(C_1+C_2+C_3)+C_1.C_2+C_1.C_3+C_2.C_3.
 \end{gather*}
\end{prop}
\begin{proof}
The formula for $\det\ke$ follows from eq. \eqref{eq: split resolution of E}, which then implies the formula for $c_1(\ke)$. Finally, using again eq. \eqref{eq: split resolution of E}, we find:
\[\prod_{i=1}^3(1+c_1(\kl)+[C_i])=(1+c_1(\kl))(1+c_1(\ke)+c_2(\ke)).\]
By replacing $c_1(\ke)$ in the above we find the desired formula for $c_2(\ke)$. 
\end{proof}

\subsection{Computational tools}

In this section we collect some tools to compute the cohomology  of (twists of)  symmetric powers of our bundle $\ke$.  Assuming $\ke$ satisfies \eqref{eq: split resolution of E}, we denote by $\tilde \ke:=\oplus_{i=1}^3\kl(C_i)$.

We will frequently be interested in the cohomology of bundles of the form $\Sym^k \ke \tensor \km$ for some line bundle $\km$. Taking symmetric powers of the sequences \eqref{eq: split resolution of E} and  \eqref{eq: Serre seq for E}, compare  \cite[Prop.~A.2.2\;d.]{Eisenbud}, we get the following useful exact sequences:
      \begin{gather}
    0 \to  \kl\tensor\Sym^{k-1}\tilde \ke  \to\Sym^k\tilde \ke \to \Sym^k\ke \to 0 \label{eq: symmetric power of nearly split sequence},\\
    0 \to  \kl(C_1)\tensor \ke  \to\Sym^2\ke \to \ki^2_{C_2\cap C_3}\tensor \kl^{\tensor 2}(2C_2 + 2C_3) \to 0 \label{eq: square of Serre sequence},\\
    0 \to  \kl(C_1)\tensor\Sym^2\ke  \to\Sym^3\ke \to \ki^3_{C_2\cap C_3}\tensor \kl^{\tensor 3}(3C_2 + 3C_3) \to 0 \label{eq: cube of Serre sequence}.
   \end{gather}
   In each case, only the injectivity of the first map has to be checked, but follows from the fact that it is an isomorphism at the generic point and the corresponding sheaves are locally free, in particular torsion-free.

   The first sequence lends itself to computations, since $\Sym^k\tilde \ke$ is by definition a direct sum of line bundles. The other two sequences gain in appeal if we note that the cohomology of twists of the ideal sheaf $\ki_{C_2\cap C_3}^k$ is often computable via the following result. 
\begin{prop}\label{prop: resolution ideal}
Let $A, B\subset S$ be curves without common components and let $Z = A \cap B$. Denoting defining equations of $A$ and $B$ by respectively $a$ and $b$, then the ideal sheaf of $Z$ is  $ \ki_Z = (a,b)$ with Koszul resolution
\[
 \begin{tikzcd}[ampersand replacement = \&]
  0 \rar \& \ko(-A-B) \rar{\mat{ b \\ - a}} \& \ko(-A) \oplus \ko(-B) \rar{(a, b)} \& \ki_Z \rar \& 0. 
 \end{tikzcd}
\]
The symmetric powers coincide with the powers of the ideal
\[ \Sym ^k \ki_Z = \ki_Z ^k = (a^k , a^{k-1} b,   \dots, b^k),\] 
and free resolutions of $\ki_Z^k$ are given as follows:
\begin{small}
\[{
 \begin{tikzcd}[ampersand replacement = \&]
  0 \rar \& \bigoplus_{i = 0}^{k-1}\ko(-(k-i) A - (i+1)B) \rar{\cdot M} \& \bigoplus_{i = 0}^k \ko(-(k-i) A - iB) \rar{(a^k ,  \dots, b^k)} \& \ki_Z^k \rar \& 0 
 \end{tikzcd}}
\]
where 
\[ M =
 \mat{ 
 b & &0  \\
 - a & \ddots\\
 & \ddots & b\\
 0 & & -a
  }.
\]
\end{small}
 \end{prop}
\begin{proof}
The sequences are evidently exact on the complement of $Z$. For every point $p\in Z$ the local equations $a, b$ form a regular sequence in the local ring $\ko_{S, p}$. In this setting, the presentation of the powers of the ideal can be read off from the Rees algebra, which coincides with the symmetric algebra by \cite{Huneke80}.
\end{proof}

\section{Triple covers}\label{sect: triple covers}
\subsection{Conventions for projective bundles}
Here we adopt the convention that for a locally free sheaf $\ke$ on $S$ 
\[ \pi\colon \IP(\ke) = \Proj_S(\Sym^* \ke)\to S\]
is the projective bundle parametrizing rank one quotient bundles of $\ke$. Thus there is a relative Euler sequence 
\begin{equation}\label{eq: Euler}
 \begin{tikzcd}
 0 \rar & \Omega_{\IP(\ke)/S} \rar & \pi^*\ke\tensor \ko_{\IP(\ke)}(-1) \rar & \ko_{\IP(\ke)}\rar & 0 
   \end{tikzcd}
\end{equation}
and, in particular, we have $\omega_{\IP(\ke)} = \pi^*(\det \ke \tensor \omega_{S})(-\rk \ke)$. 

\subsection{Triple covers after Miranda, Casnati and Ekedahl}
We recall the construction of Casnati and Ekedahl \cite{CE96}, and Miranda \cite{Miranda}. Let $S$ be a smooth projective surface and $\ke$ a rank two vector bundle on $S$. On the projective bundle $\pi\colon \IP= \IP(\ke)\to S$ we have the relative tautological bundle $\ko_{\IP}(1)$ and, via pushforward, a natural isomorphism
\[ \Phi_3 \colon H^0(S, \Sym^3\ke\tensor \det \ke^\vee) \isom H^0(\IP, \pi^*\det \ke^\vee(3)).\]

\begin{defin}
 We say a section $\eta \in H^0(S, \Sym^3\ke\tensor \det \ke^\vee)$ has the right codimension at $x\in S$ if $\Phi_3(\eta)$ does not vanish identically along the fibre $\inverse \pi(x)$.
 We say $\eta$ has the right codimension if it has the right codimension at every point of $S$. 
\end{defin}
If $\eta$ has the right codimension at every point, then the vanishing  subscheme of $\Phi_3(\eta)$ defines a Gorenstein triple cover, as in each fibre it cuts out a Gorenstein subscheme of length three. More precisely, we have: 

\begin{thm}[Casnati-Ekedahl, Miranda]\label{thm: CE}
Let $S$ be a smooth projective surface.
The above construction gives a bijective correspondence beween connected Gorenstein triple covers $f\colon T \to S$ with  $\ke = \left( f_*(\ko_T) / \ko_S\right)^\vee$  and rank two vector bundles $\ke$ on $S$ such that  $H^0(S, \ke^\vee) = 0$ together with a section $\eta \in H^0(S, \Sym^3 \ke \tensor \det \ke ^\vee)$ having the right codimension up to non-constant multiples. 

Moreover, if $\Sym^3 \ke \tensor \det \ke ^\vee$ is globally generated, then a general section $\eta$ defines a smooth triple cover. 
 \end{thm}

The bundle $\ke$ is called \textit{the Tschirnhausen bundle} of the triple cover.
 \begin{defin}
  We say a rank $2$ bundle $\ke$ is admissible if it is the Tschirnhausen bundle of a triple cover $f\colon T\to S$ with $T$ smooth and connected. 
   \end{defin}
 If $f\colon T\to S$ is a triple cover then, by relative duality, for every vector bundle $\kf$ on $T$ we have
 \[ f_*\shom_T(\kf, \omega_T) \isom \shom_S(f_*\kf, \omega_S).\]
Either from this, or by construction  from \cite{CE96}, we see that 
 \begin{equation}\label{eq: pushforwardKT}
\omega_T = \ko_{\IP}(1) \tensor\pi^*\omega_S |_T \text{ and }   f_*\omega_T \isom  \shom_S(f_*\ko_T, \omega_S)\isom \omega_S \oplus( \ke\tensor \omega_S).
 \end{equation} 
 
As an application of this equation, we show how positivity of the Tschirnhausen bundle forces $T$ to be minimal of general type. 
By \eqref{eq: pushforwardKT} the canonical bundle  $\omega_T$ is ample if $\ko_{\IP(\ke)}(1)\tensor \pi^*\omega_S\isom \ko_{\IP(\ke\tensor \omega_S)}(1)$ is ample. By definition, this is the case if and only if the vector bundle $\ke\tensor \omega_S$ is ample, compare \cite{LazarsfeldII}.
 While ampleness of bundles in general is subtle, we formulate a useful sufficient criterion.
 \begin{prop}\label{prop: general criterion min gen type}
 Let $f\colon T \to S$ be a triple cover with Tschirnhausen bundle $\ke$ and let $\kl$ be a globally generated ample line bundle on $S$. If $\ke\tensor \omega_S \tensor \kl^\vee$ is globally generated on $S$, then $T$ is minimal of general type. 
 \end{prop}
 \begin{proof}
  As mentioned above, it is enough to prove that $\ke\tensor \omega_S$ is ample, which follows from \cite[Example~6.1.5]{LazarsfeldII}.
 \end{proof}
 \subsection{Warm-up: split Tschirnhausen bundle}
Consider as a warm-up the case of a triple cover $f \colon T \to S$, where the Tschirnhausen bundle is split, that is, $\ke\isom \kl_1\oplus \kl_2$. This includes the case of a $\IZ/3\IZ$-Galois cover, compare \cite{pardini91}.

The first condition of the admissibility is easy to check:
\begin{lem}\label{lem: split connectedness condition}
 We have $h^0(S, \ke^\vee ) =0$ if and only if $h^0(S, \kl_1^\vee) = h^0(S, \kl_2^\vee) = 0$.
 \end{lem}

 The bundle containing the sections defining a triple cover is
 \[ \Sym^3\ke \tensor \det \ke^\vee = \bigoplus_{ k = 0}^3 \kl_1^{k-1} \tensor \kl_2^{3-k-1}\]
 and thus some necessary and sufficient conditions on the existence of a smooth cover can be easily found.
 \begin{prop}\label{prop: split case globally generated}
  If for all $0\leq k\leq 3$ the line bundles $ \kl_1^{k-1} \tensor \kl_2^{3-k-1}$ are globally generated, then there exists a smooth triple cover $f\colon T \to S$ with Tschirnhausen bundle $\ke$. 
 \end{prop}
\begin{proof}
 Clear from Theorem  \ref{thm: CE}.
\end{proof}

Let us denote the fibre coordinates on the $\IP^1$-bundle by $v, w$. Then since the bundle is globally split, any defining section can be globally written as
  \begin{equation}\label{eq: etatilde}
\eta = \sum_{k=0}^3 f_k v^{k} w^{3-k}
\end{equation}
with $f_{k}\in H^0(S,\kl_1^{k-1} \tensor \kl_2^{3-k-1})$. 
In order to see if this is at least globally irreducible, one would have to compute the discriminant, which is akward to handle in this generality. At least we have an easy criterion for the equation to be divisible by one of the variables.
\begin{lem}\label{lem: split case necessary}
 If $\ke = \kl_1\oplus \kl_2$ is admissible, then 
 \[H^0(S, \kl^3_1\tensor\kl_2^\vee)\neq 0 \ \text{ and } H^0(S, \kl^3_2\tensor\kl_1^\vee)\neq 0 .\] 
\end{lem}

These easy conditions suffice to give a complete answer in particular cases.
\begin{exam}\label{exam: split triple planes}
 Let $S = \IP^2$ and $\ke = \ko(a_1) \oplus \ko( a_2)$ with $a_1\geq a_2>0$ satisfying the condition of Lemma \ref{lem: split connectedness condition}. Then the least positive line bundle summand of $\Sym^3\ke \tensor \det \ke^\vee $ is 
 $\ko_{\IP^2}(2a_2-a_1)$, so by Lemma \ref{lem: split case necessary} we have $2a_2-a_1 \geq 0 $ if $\ke$ is admissible. But since on $\IP^2$ any line bundle of non-negative degree is globally generated, this condition is already sufficient for admissibility by Proposition \ref{prop: split case globally generated}. 

 In total, we have proved that $\ke = \ko(a_1) \oplus \ko( a_2)$ is admissible if and only if
 \[ 0 <a_2 \leq a_1 \leq 2a_2.\]
 
We will discuss what kind of surfaces can arise as triple planes with split Tschirnhausen bundle in Section \ref{sect: triple planes}.
 \end{exam}

 \section{Special triple covers}\label{sect: special triple covers}
  \begin{defin}
  We call a triple cover $f \colon T \to S$ \emph{special}  if the Tschirnhausen bundle $\ke$ is nearly split in the sense of Proposition \ref{prop: nearly split E}.
 \end{defin}
 We will now gather some information on such covers and start with the following remark.
 \begin{rem}\label{rem: IP and tilde IP}
  Let $\ke$ be  a nearly split bundle associated to $(\kl, C_1, C_2, C_3)$. Then the surjection  
  \[\tilde \ke := \kl(C_1)\oplus \kl(C_2) \oplus \kl(C_3)\onto \ke\]
  defines an embedding of projective bundles 
    \[
  \begin{tikzcd}
   \IP := \IP(\ke)\arrow[hookrightarrow]{rr}\arrow{dr}[swap]{\pi} && \tilde \IP :=\IP(\tilde \ke)\arrow{dl}{\tilde \pi}\\ & S
  \end{tikzcd}
  \]
  with the property that $\ko_{\tilde \IP}(1)|_{\IP} = \ko_{\IP}(1)$.
  
Let us denote a defining equation for $C_i $ by $\gamma_i$ and let $u,v,w$ be the corresponding sections of $\ko_{\tilde\IP}(1)\otimes\tilde\pi^*\kl^\vee(-C_i)$ with $i \in \{1,2,3\}$ defining $\IP(\kl(C_j)\oplus\kl(C_k))\subset\tilde\IP$ via the natural composition
 \[ \tilde\pi^*\kl(C_i)\rightarrow\tilde\pi^*\tilde\ke\rightarrow\ko_{\tilde\IP}(1).\]
 
 Then the equation for $\IP\subset\tilde\IP$ is
 \begin{equation}\label{eq: eq for IP} 
\gamma_1 u + \gamma_2 v + \gamma_3 w = 0, 
 \end{equation}
 which is a section of $\ko_{\tilde \IP}(1) \tensor \tilde \pi^*\kl^\vee$.
 \end{rem}

 If $\ke$ is admissible, then together with the triple cover we get a diagram
  \[
  \begin{tikzcd}
T \arrow{dr}[swap]{f}\rar[hookrightarrow] &  \IP\arrow[hookrightarrow]{r}\arrow{d}{\pi} \rar & \tilde \IP \arrow{dl}{\tilde \pi}\\
  & S
 \end{tikzcd}
 \]
 where $T$ is defined in $\IP$ by a section  $\eta \in H^0(\IP, \pi^*\det \ke^\vee(3))$.
The following definition will prove to be useful later.
\begin{defin}\label{def: ci cover}
 We say that the special triple cover $f\colon T\to S$ is a global complete intersection (gci) if $T = \tilde T \cap \IP\subset \tilde \IP$, that is, if the defining equation $\eta \in H^0(\IP, \pi^*\det \ke^\vee(3))$ lifts to an equation $\tilde \eta \in H^0(\tilde \IP, \tilde\pi^*\det \ke^\vee(3))$.
\end{defin}

 The invariants of triple planes have been computed also by Miranda \cite{Miranda}, but we include below a proof for the convenience of the reader. 
 \begin{cor}\label{cor: invariants}
  In the above situation, the invariants of $T$ are 
  \begin{align*}
   K_T^2 & =3K_S^2+4K_S.(2\kl+\sum_iC_i)+(5\kl^2+5\kl.\sum_iC_i+2\sum_iC_i^2+\sum_{i<j}C_iC_j) \\
   \chi(T) &=3\chi(S)+\kl^2+\kl.K_S+\kl.\sum_iC_i+\frac12 \sum_i(C_i^2+K_S.C_i). \\
   \chi(2K_T) &  =
   3(\chi(S)+K_S^2)+9K_S.\kl+6\kl^2\\ & \qquad +6\kl.\sum_iC_i+\frac92 K_S.\sum_iC_i+\frac{5}{2}\sum_iC_i^2+\sum_{i<j}C_i.C_j \\
   p_g(T) & = h^0 (S, \ke\tensor \omega_S)+p_g(S) \\
   q(T) &  = h^1(S, \ke\tensor \omega_S)+q(S)\end{align*}
  
The plurigenera of $T$ are computed as $P_2(T)= h^0(\Sym^2\ke\tensor \omega_S^2)$ and 
  \[P_m(T)\geq h^0(S,\Sym^m\ke\otimes\omega_S^m)-h^0(S,\Sym^{m-3}\ke\otimes\det\ke\otimes\omega_S^m)\]
  with equality if 
  $h^1(S,\Sym^{m-3}\ke\otimes\det\ke\otimes\omega_S^{m})=0$.
 \end{cor}
\begin{proof}
  Let us put $a=c_1(\omega_S)\in H^2(S,\IR)$ and $u:=c_1(\ko_\IP(1))\in H^2(\IP,\IZ)$. Then we recall that we have the relation in $H^\bullet(\IP,\IZ)$
  \[ \pi^*c_2(\ke)-\pi^*c_1(\ke).u+u^2=0\]
  implying
  \[u^2=u.\pi^*c_1(\ke)-\pi^*c_2(\ke), \ \ u^3=u.\pi^*(c_1^2(\ke)-c_2(\ke)).\]
  Using this and \eqref{eq: pushforwardKT}, we compute:
  \begin{align*}
  K_T^2=&\pi^*\omega_S(1).\pi^*\omega_S(1).  \pi^*\det \ke^\vee (3)\\
  =&(\pi^*a+u)^2(-\pi^*c_1(\ke)+3u)\\
  =&u.\pi^*(-2c_1(\ke).a+3a^2)+u^2\pi^*(-c_1(\ke)+6a)+3u^3\\
  =&u.\pi^*\left(-2c_1(\ke).a+3a^2-c^2_1(\ke)+6a.c_1(\ke)+3c_1^2(\ke)-3c_2(\ke)\right)\\
  =&u.\pi^*\left(3a^2+4a.c_1(\ke)+2c_1^2(\ke)-3c_2(\ke)\right).
  \end{align*}
By applying the push forward map, using that $\pi_*u=1$ and Proposition \ref{prop: Cherncl}, we find:
\begin{align*}
K_T^2&=3K_S^2+4K_S.c_1(\ke)+2c_1^2(\ke)-3c_2(\ke)\\
&=3K_S^2+4K_S.(2\kl+\sum_iC_i)+(5\kl^2+5\kl.\sum_iC_i+2\sum_iC_i^2+\sum_{i<j}C_iC_j).  
\end{align*}

We now compute the holomorphic Euler characteristic using \eqref{eq: split resolution of E} and Riemann-Roch:
\begin{align*}
\chi(T)&=\chi(\ko_T) = \chi(S) + \chi(\ke^\vee) \\
&=\chi(S)+\sum_i\chi(\kl^\vee(-C_i))-\chi(\kl^\vee)\\
&=3\chi(S)+\frac12\left(\sum_i(\kl+C_i).(\kl+C_i+K_S)-\kl.(\kl+K_S)\right)\\
&=3\chi(S)+\kl^2+\kl.K_S+\kl.\sum_iC_i+\frac12 \sum_i(C_i^2+K_S.C_i).
\end{align*}

The formulas for $p_g$ and $q$ follow from \eqref{eq: pushforwardKT}. In order to compute $P_m$ for $m\geq 2$, note that since $T\subset \IP$ is given by a section of $\pi^*\det\ke^\vee(3)$, we have the following exact sequence:
\begin{equation*}
0\to\pi^*\det\ke\otimes\ko_{\IP}(m-3)\to\ko_{\IP}(m)\to\ko_T(m)\to 0.
\end{equation*}
Since for $m\geq 2$ we have $R^1\pi_*\ko_{\IP}(m-3)=0$, after applying $\pi_*$ to the above and using the projection formula, we find
the exact sequence:
\begin{equation*}
0\to\det\ke\otimes\Sym^{m-3}\ke\to\Sym^m\ke\to f_*\ko_T(m)\to 0
\end{equation*}
where we use the convention $\Sym^{-1}\ke=\pi_*\ko_\IP(-1)=0$.
Finally, in order to deduce $P_m$, one uses the above exact sequence, together with (ii) in the above proposition, giving
\begin{align*}
P_m&=h^0(S,f_*\omega^m_T)=h^0(S,\omega^m_S\otimes f_*\ko_T(m))\\
&\geq h^0(S,\Sym^m\ke\otimes\omega_S^m)-h^0(S,\Sym^{m-3}\ke\otimes\det\ke\otimes\omega_S^m).
\end{align*}
In particular, $P_2 = h^0(\Sym^2\ke\tensor \omega_S^2)$. 
  \end{proof}
 \begin{cor}\label{cor: minimal general type}
Let $f\colon T \to S$ be a special triple cover with Tschirnhausen bundle $\ke$. Then $T$ is minimal of general type if and only if.  
\[K_T^2>0,\; P_2(T) = h^0(\Sym^2\ke\tensor \omega_S^2)>0, \; h^1(T, \omega_T^2) = h^1(\Sym^2\ke\tensor \omega_S^2)=0.\]
 \end{cor}
 \begin{proof}
The conditions are clearly satisfied for a minimal surface of general type. Conversely, the first two conditions ensure that $T$ is of general type. The last vanishing characterises minimality by Riemann-Roch.
 \end{proof}
 
 \section{Some necessary and sufficient conditions for admissibility}\label{sect: suffcient and necessary conditions}

We now address the generally difficult question, which nearly split bundles are admissible, that is, occur as Tschirnhausen bundles of a smooth triple cover $f\colon T \to S$.  
We continue to use the same notation, see Remark \ref{rem: IP and tilde IP}.

We start by settling when a potential cover would be connected. 
 \begin{lem}[connectedness condition]\label{lem: condition connectedness}
If $H^0(S, \ke^\vee) =0$, then
\begin{align}\label{eq:conn}
 H^0(S,\kl^\vee(-C_i-C_j))&=0, \ \ 1\leq i<j\leq 3.
 \end{align}
 Conversely, if \eqref{eq:conn} holds and also 
 \[H^0(S,\kl^\vee(-C_1-C_2)|_{C_3})=0 \text{ or } H^1(S, \kl^\vee(-C_1-C_2-C_3)) = 0,\] 
 then $H^0(S,\ke^\vee)=0$. 
  \end{lem}
\begin{proof}
 Dualising \eqref{eq: split resolution of E} and taking cohomology, we get the exact sequence
 \[ 0 \to H^0(S,\ke^\vee) \to \bigoplus H^0(S,\kl^\vee(-C_i))\to H^0(S,\kl^\vee).\]
So clearly $H^0(S,\ke^\vee)=0$ iff the last map is injective. In particular, we must have that the last map in the exact sequence
\[ 0 \to H^0(S,\kl^\vee(-C_i-C_j)) \to  H^0(S,\kl^\vee(-C_i))\oplus H^0(S,\kl^\vee(-C_j))\to H^0(S,\kl^\vee)\]
is injective for any $1\leq i<j\leq 3$, which is equivalent to \eqref{eq:conn}. 

Supposing now that $\eqref{eq:conn}$ holds, we have that $H^0(S,\ke^\vee)=0$ iff the image of $H^0(S,\kl^\vee(-C_3))$ in $H^0(S,\kl^\vee)$ does not intersect the image of $H^0(S,\kl^\vee(-C_1))\oplus H^0(S,\kl^\vee(-C_2))$. This is further equivalent to the map $\psi$ in the following  commutative diagram to be injective:
\begin{small}
\begin{equation*}
{
\begin{tikzcd}
   & 0 \dar\rar & H^0(\kl^\vee(-C_3)) \dar{\gamma_3} \\
0 \dar\rar & H^0(\kl^\vee(-C_1))\oplus H^0(\kl^\vee(-C_2)) \rar{(\gamma_1,\gamma_2)} \dar\arrow[rd, "\psi"] & H^0(\kl^\vee)\dar\\
H^0(\kl^\vee(-C_1-C_2)|_{C_3})\rar\dar & H^0(\kl^\vee(-C_1)|_{C_3})\oplus H^0(\kl^\vee(-C_2)|_{C_3}) \rar & H^0(\kl^\vee|_{C_3})\\
H^1(\kl^\vee(-C_1-C_2-C_3)) & &
\end{tikzcd}}
 \end{equation*}
 \end{small} 
 Hence, if $H^0(S,\kl^\vee(-C_1-C_2)|_{C_3})=0$, then $\psi$ is injective, as a composition of two injective maps. Finally, $H^0(S,\kl^\vee(-C_1-C_2)|_{C_3})=0$ as soon as $H^1(S,\kl^\vee(-C_1-C_2-C_3))=0$.
\end{proof}

Before we go on, let us distinguish several ways in which $\ke$ can be admissible.
\begin{defin}
 Let $\ke$ be an nearly split bundle. We say that $\ke$ is 
 \begin{enumerate}
  \item trivially admissible if $\Sym^3\tilde \ke\tensor \det \ke ^\vee$ is globally generated;
  \item gg-admissible if $\Sym^3\ke\tensor \det \ke ^\vee$ is globally generated; 
  \item gci-admissible if there exists a smooth global complete intersection special cover in the sense of Definition \ref{def: ci cover}.
 \end{enumerate}
\end{defin}

\begin{prop}\label{prop: implications admissibility}
 The following implications hold:
 \[\begin{tikzcd}
  \text{ trivially admissible} \rar[Rightarrow] \dar[Rightarrow] & \text{gg-admissible}\dar[Rightarrow] \\
  \text{gci-admissible} \rar[Rightarrow] & \text{admissible} \lar[Rightarrow, bend left]{\text{if }H^1(\tilde \IP, \tilde \pi^*(\det \ke^\vee \tensor \kl) (2))= 0}
 \end{tikzcd}.
\]
\end{prop}
\begin{proof}
  Since there is a surjection $\Sym^3\tilde \ke \tensor \det \ke^\vee \onto \Sym^3 \ke \tensor \det \ke^\vee $, the latter is globally generated if the former is so. Thus trivially admissible implies gg-admissible, which in turn implies admissible by Theorem \ref{thm: CE}. 
  
  To understand gci-admissibility, which implies admissibility by definition, we look at the twisted  restriction sequence 
 \[ 0 \to \tilde \pi^*(\kl\tensor \det \ke^\vee) (2) \to \tilde \pi^* \det \ke ^\vee (3) \to \pi^* \det \ke ^\vee (3) \to 0. \]
If $\ke$ is trivially admissible, then the bundle in the middle is globally generated and by Bertini \cite{Bertini} there is a smooth intersection $T = \tilde T \cap \IP$, so $\ke $ is gci-admissible.  

Clearly, if the given $H^1$ vanishes, then every $\eta \in H^0(\IP, \pi^* \det \ke ^\vee (3))$ lifts to a section defined on $\tilde\IP$, so the converse implication in the lower row holds. 
\end{proof}

For later reference we  the following. 
\begin{lem}\label{lem: consequences of splitting}
 We have: 
 \begin{enumerate}
  \item The bundle $\ke$ is trivially admissible if and only if for all $a_1 + a_2 + a_3 = 3$ the line bundle $\kl\left(\sum_i (a_i-1) C_i\right)$ is globally generated. 
  \item If for all  $a_1 + a_2 + a_3 = 3$ we have $H^1(S, \kl\left(\sum_i (a_i-1) C_i\right))=0$, then every triple cover with Tschirnhausen bundle $\ke$ is gci. 
 \end{enumerate}
\end{lem}
\begin{proof}
 This comes from our definition of bundle $\tilde \ke$, defined in Remark \ref{rem: IP and tilde IP}, and hence all its symmetric powers are direct sums of line bundles and both global generation and vanishing of $H^1$ can be tested on summands. 
\end{proof}
If we are not in the fortunate situation of trivial admissibility, we can give some necessary criteria. 
 \begin{lem}[Necessary conditions for gg-admissibility]\label{lem: nec for easy existence}
 If $\, \Sym^3\ke \otimes \det\ke^\vee \,$ is globally generated with $h^0(\ke^\vee) = 0$, hence admissible, then for all curves $D\subset S$, the bundle $\Sym^3\ke\otimes\det\ke^\vee|_D$ is globally generated. In particular, for all $\{i,j,k\} = \{1,2,3\}$ the bundle 
 \[\left(\det\ke^\vee \tensor \Sym^3\left( \kl(C_i+C_j) \oplus \kl(C_k)\right) \right)|_{C_k}\]
  is globally generated. 
  
 \end{lem}
\begin{proof}
 The first assertion is obvious, while for the second one we just need to note what happens to \eqref{eq: split resolution of E} upon restricting to $C_k$. 
\end{proof}

We include a general necessary condition for existence.
\begin{prop}[Necessary condition for an irreducible cover]\label{prop: necessary condition irreducibility}
Let $\eta$ be a section in $H^0(S, \Sym^3(\ke) \tensor \det \ke^\vee)$. If $\eta$ defines an irreducible triple cover $f\colon T \to S$, then for all $i\neq j\neq k\neq i$
\[ H^0(S, \kl(2(C_j+C_k)-C_i)\tensor \ki^3_{C_j\cap C_k})\neq 0.\]
\end{prop}
\begin{proof}
The irreducibility of $T$ is governed by the irreducibility of the cubic equation $\eta$ at the generic point, which is not easy to control in general. 

 However, if we restrict the sequence \eqref{eq: Serre seq for E}   to  the open subset $U = S\setminus \cup_{l=1}^3C_l$, then, since $\kl|_{U} \to \kl(C_j)|_{U}$ is an isomorphism, it is split exact and gives an isomorphism $\ke|_U \isom \kl(C_i)|_U\oplus\kl(C_j+C_k)|_U$.
 
Also, as the projective bundle $\IP|_U \to U$ is trivial, we can write the section 
 \[ \eta|_U = \sum_{i = 0}^3 \eta_i u^iv^{n-i}.\]
With this, the natural restriction map 
\[H^0(S, \Sym^3(\ke) \tensor \det \ke^\vee) \to H^0(S, \kl(2(C_j+C_k)-C_i)\tensor \ki^3_{C_j\cap C_k})\]
induced by the third power of \eqref{eq: Serre seq for E} corresponds to the restriction to the $\eta_0$ components. If this component vanishes, then over $U$ the equation is divisible by $u$ and $T$ is reducible for sure. 
\end{proof}
In case of special building data, we can be much more precise.

\begin{thm}[gci-admissibility for commensurable data]
\label{thm: gci existence}
Let $A$ be a divisor on a smooth surface $S$ such that $|A|$ has no base points. 
Fix integers  $d$ and $c_1\geq c_2 \geq c_3 \geq 1$ and let $C_i\in |c_i A|$ be such that the triple intersection is empty.  
Let $\ke$ be nearly split associated to the data $(\ko_S(d A), C_1, C_2, C_3)$. 
\begin{enumerate}
\item If $d+c_3\geq c_1+c_2-c_3$, then $\ke$ is trivially admissible. 
\item If $c_1+c_2-c_3>d+c_3\geq c_1$, then $\ke$ is  gci-admissible but not trivially admissible.
\item If $c_1> d+c_3 \geq c_2$ and $d -c_1 +2c_2 - c_3 \geq 0$, 
 then $\ke$ is  gci-admissible if and only if $C_3$ is smooth at every point of $C_2 \cap C_3$, 
\item In all other cases, the bundle  $\ke$ is not gci-admissible. 
\end{enumerate}
\end{thm}
\begin{proof}
Using the notation from Remark \ref{rem: IP and tilde IP}, the splitting of $\Sym^3\tilde \ke$, and the formula for $\det \ke$, every section $\tilde\eta\in H^0(\tilde \IP,   \tilde \pi^* \det \ke ^\vee (3))\isom  H^0(S, \Sym^3\tilde \ke \tensor \det \ke ^\vee)$ can be written as 
\begin{equation}\label{eq: tildeT}
\tilde \eta = \sum_{a+b +c= 3} f_{abc} u^av^bw^c
\end{equation}
with $f_{abc}\in H^0\left(S, \ko_S\left((d+ (a-1)c_1 + (b-1) c_2 + (c-1) c_3)A\right)\right)$; we denote its restriction to $\IP$ with $\eta = \tilde \eta|_\IP$.  

We will use repeatedly that for any integer $m$, the  $\ko_S(mA)$ is  globally generated if $m\geq 0$ and has no sections if $m<0$. 
The lowest multiple of $A$ occurs for the coefficient of $w^3$, namely
 \[f_{003} \in H^0\left(S, \ko_S\left((d-c_1 - c_2+ 2c_3)A\right)\right),\] 
 so if $d+c_3\geq c_1+c_2-c_3$, then $\ke$ is trivially admissible, which proves $(i)$.

For the other items, consider the inequalities
\[ \underbrace{d -c_1 +2c_2 - c_3}_{\geq 0 \text{ in  $(iii)$}} \geq d -c_1 +c_2 \geq\underbrace{ d -c_1 + c_3}_{\geq 0 \text{ in  $(ii)$}} \geq \underbrace{d -c_1 -c_2 + 2c_3}_{\geq 0 \text{ in  $(i)$}}, \]
involving the multiples of $A$ corresponding to the coefficients of $v^3$, $v^2w$, $vw^2$ and $w^3$, as well as the inequalities  
\[\underbrace{d -  c_2 + c_3}_{\geq 0 \text{ in  $(iii)$}} \geq\underbrace{ d -c_1 + c_3}_{\geq 0 \text{ in  $(ii)$}}\]
involving the multiples of $A$ corresponding to the coefficients of $uw^2$ and $vw^2$.

We see that the bundles containing the coefficients $f_{abc}$ in $(iii)$ are less positive than the ones  in $(ii)$.

We first show $(iv)$ by proving that $\ke$ is not gci-admissible if either one of the conditions in $(iii)$ is not met. 

If $ d -c_1 +2c_2 - c_3<0$ then, by the above, the equation $\tilde \eta$ is globally divisible by $u$ and there cannot exist an irreducible gci-cover.

Now assume $d+c_3-c_2<0$. Then we have also 
\[ d-c_1-c_2+2c_3\leq d-c_1+c_3\leq d+c_3-c_2<0\]
so that only monomials of degree $0$ or $1$ in $w$ appear in \eqref{eq: tildeT}.
Assume that a gci triple cover $T$ is defined by restriction of an equation $\tilde \eta$ to $\IP$. 

Let us restrict to $\tilde U:=\{w\neq 0\}|_{S\setminus C_2}\subset\tilde\IP$ and replace $u$ by $\frac uw$ and $v$ by $\frac vw$,  so that \eqref{eq: eq for IP} becomes over $\tilde U$:
\[v=\alpha u+\beta, \ \alpha=-\frac{\gamma_1}{\gamma_2}, \ \beta=-\frac{\gamma_3}{\gamma_2}.\]

Then on $U:=\tilde U\cap \IP$, $\eta$ becomes:
\[ \eta=\sum_{a+b=3}f_{ab0}u^a(\alpha u+\beta)^b+\sum_{a+b=2}f_{ab1}u^a(\alpha u+\beta)^b\]
an equation which vanishes at the curve defined by $u = \beta = 0$, lying over $C_3$,  with multiplicity at least two. Thus $T$ cannot by smooth, indeed not even normal, and $\ke$ is not gci-admissible. 
This concludes the proof of $(iv)$. 

To address $(ii) $ and $(iii)$ we now assume 
that $d -c_1 +2c_2 - c_3 \geq 0 $, that is, we can have $v^3$ with non-zero coefficient in $\tilde \eta$.  

We may also assume that we are not in case $(i)$, that is, $c_1+c_2  -c_3 > d+c_3 $ and $w^3$ does not appear in $\tilde \eta$.  We claim that the base locus of the linear system on $\IP$ is then the curve 
\[ B:=\Bs\left(| \ko_\IP(3)\otimes\pi^*\det \ke^\vee|\right) = \{ u=v = \gamma_3 = 0\},\]
lying over $C_3$. 
Indeed, both equation vanish along $B$, because $\tilde \eta$ does not contain the monomial $w^3$ with a non-zero coefficient. At a point outside $B$, one monomial of the form $f_{300}u^3$ or $f_{030}v^3$ does not vanish, since both $f_{030}$ and $f_{030}$ can be chosen in base-point free linear systems. 

By Bertini \cite{Bertini} and compactness of $B$, it follows that $\ke$ is gci-admissible if and only if for any $z\in B$, there is a non-zero section $\tilde \eta$ cutting out a surface on $\IP$ which is smooth at $z$. 

For the proof of $(iii)$, suppose that $d-c_1+2c_2-c_3\geq 0$ and $c_1> d+c_3 \geq c_2$, so 
\[d-c_1-c_2+2c_3\leq d-c_1+c_3<0 \] 
and the monomials $vw^2$ and $w^3$ do not appear in \eqref{eq: tildeT}.

Since $w\neq 0$ on $B$, we can again restrict to $w\neq 0$ and divide by $w$. Let us first take $z\in B\cap\tilde U$ over $x\in C_3\setminus C_2$.  
 We do the same as in $(iv)$ and see that $\eta$ becomes over $U$:
\[ \eta=\sum_{a+b=3}f_{ab0}u^a(\alpha u+\beta)^b+\sum_{a+b=2}f_{ab1}u^a(\alpha u+\beta)^b+f_{102}u.\]
Its first order term is thus $f_{102}u$, and since $f_{102}$ is a section of a globally generated line bundle by assumption, $\eta$ can be taken smooth at such points $z$.

Let us now take $z\in B$ over $x\in C_3\cap C_2$ and restrict to the open set $\tilde V:=\{w\neq 0\}|_{S\setminus C_1}$, where \eqref{eq: eq for IP} becomes
\[ u=\mu v+\nu, \ \mu=-\frac{\gamma_2}{\gamma_1}, \ \nu=-\frac{\gamma_3}{\gamma_1}.\]
Then on $V:=\tilde V\cap \IP$, the lowest order term of $\eta$ at $z$ is the lowest order term of 
\[  f_{102}(\mu v+\nu). \]
Since $\mu v$ is at least of order two and $f_{102}$ lives in a globally generated line bundle, we find that  $\eta$ can be chosen smooth at $z$ precisely when $\nu$ is of order $1$ at $x$. This shows the assertion.  

For $(ii)$ we have $c_1+c_2-c_3>d+c_3\geq c_1$ and only the monomial $w^3$ does not appear in the expression of $\tilde\eta$. As before, in order to ensure that $\ke$ is gci-admissible, it is enough to show that at any $z\in B$ one can find a section $\eta$ smooth at $z$.

By taking first $z\in B$ over $x\in C_3\setminus C_2$, we see just like before that the lowest order term of $\eta$ at $z$ is given by the first order term of
\[ f_{102}u+f_{012}(\alpha u+\beta).\]
Since $f_{102}+\alpha f_{012}$ can be chosen not to vanish at $x$, $\eta$ can be chosen smooth at $z$. Similarly, at $z\in B$ over $x\in C_3\cap C_2$, the first order term of $\eta$ contains the monomial $(f_{102}\mu+f_{012})v$ which can be made non-vanishing at $z$. Hence $\ke$ is indeed gci-admissible. 
\end{proof}

\begin{exam} \label{exam: covers of K3}
 We want to point out one exceptionally simple special case of Theorem \ref{thm: gci existence}, namely when $c_1= c_2 = c_3 = c\geq 1$ and $d\geq 0$.

Let $S$ be a smooth surface and $A$ a globally generated line bundle. Let $H$ be a line in $\IP^2$.  Writing  $\tilde \IP \isom S \times \IP^2$ and   denoting by $p$ the projection to the second factor, the $\IP^1$-bundle $\IP$ is defined by a general element of $ |c\pi^*A + p^*H|$. 
 
 Thus a gci triple cover $f\colon T\to S$ with this building data is a complete intersection of
 \[ \IP\in  |c\tilde \pi^*A + p^*H| \text{ and } \tilde T \in |d\tilde \pi^*A + 3p^*H|,\]
 which is obviously admissible by Bertini. Its canonical bundle is 
 \[\omega_T = f^*\omega_S((c+d)A)\tensor p^*\ko_{\IP^2}(1).\]

Assume that $S$ is a K3 surface, a case recently considered by Garbagnati and Penegini in \cite{GP21}.  Since $A$ is assumed to be globally generated, we have $A^2\geq 0$, thus  
 \[ K_T^2 = (5d^2+ 15cd + 9c^2) A^2\geq 0 
\, .
 \]
 We see that if $A^2 >0$,  then $K_T^2>0$ and $p_g(T)\geq p_g(S)>0$, so $T$ is of general type. If $A$ is even ample, then the canonical $\omega_T$ is clearly ample too, and $T$ is in fact minimal of general type.
 
 If $A^2 = 0$ but $A$ is not trivial, then $|A|$ induces a fibration over a curve $\phi_{|A|}\colon S \to B$ and it is easy to check that the triple cover $T \to S$ is then the pullback of a triple cover $B' \to B$. 
 
Therefore, via this construction, we obtain either surfaces of general type, by choosing $A$ ample, or properly elliptic surfaces if $A$ induces an elliptic fibration. 
 \end{exam}

\section{Asymptotic considerations}\label{sect: aymptotics}
Let us fix an ample divisor $A$ on the smooth surface $S$. Assume that $\ke$ is an nearly split bundle on $S$ associated to $(\kl, C_1, C_2, C_3)$. 
It is plausible that $\ke$ will become admissible if we make the data suffciently positive. 

Below, we give some indication how this can be done but also point out that one has to keep a certain balance for admissibility. 

\begin{prop}
Consider the bundle $\ke'_m = \ke(mA)$, nearly split with associated data $(\kl(mA), C_1, C_2, C_3)$.
 Then for $m\gg 0$ the bundle $\ke'_m$ is gg-admissible. If $T_m$ is a general triple cover with Tschirnhausen bundle $\ke'_m$, then for $m\gg 0$ it is minimal of general type and asymptotically the slope becomes
\[ \lim_{m\to \infty } \frac{K_{T_m}^2}{\chi(T_m)} = 5.\]
 \end{prop}
\begin{proof}
First note that for $m\gg 0$ the bundle
\[ \Sym^3\ke'_m\tensor \det {\ke'}_m^\vee =\Sym^3\ke\tensor \det \ke^\vee (mA)\]
is globally generated, hence $\ke'_m$ is gg-admissible.

Substituting into the formulas of Corollary \ref{cor: invariants} we see that 
 \begin{align*}
 K_{T_m}^2 & = 5m^2A^2 + \text{lower order terms}\\
\chi(T_m) & = m^2A^2 + \text{lower order terms}.
\end{align*}
So if we choose $m$ large enough such that in addition  $K_{T_m}^2 >0$ and $\Sym^2\ke'_m \tensor \omega_S^2 = \Sym^2\ke \tensor \omega_S^2(2mA)$ has no higer cohomology and at least one section, possibly by Serre vanishing, then we conclude from Corollary \ref{cor: minimal general type}.  
\end{proof}

\begin{prop}
Assume that $\ke$ is trivially admissible, that is, $\Sym ^3\tilde \ke\tensor \det \ke^\vee$ is globally generated, and  consider the bundle $\ke'_m$, nearly split with associated data $(\kl, C_1', C_2', C_3')$, where $C_i'\in |C_i +mA|$ is sufficiently general.
Then for any $m\geq 0$
  the bundle $\ke'_m$ is gg-admissible. If $T_m$ is a general triple cover with Tschirnhausen bundle $\ke'_m$, then for $m\gg 0$ it is minimal of general type and asymptotically the slope becomes
\[ \lim_{m\to \infty } \frac{K_{T_m}^2}{\chi(T_m)} = 6.\]
 \end{prop}
\begin{proof}
Note that $\Sym ^3\tilde \ke\tensor \det \ke^\vee = \Sym ^3\tilde \ke_m'\tensor \det {\ke_m'}^\vee$ for any $m$, so the bundle remains trivially admissible. 

We can use the exact sequence \eqref{eq: symmetric power of nearly split sequence} for $\ke'_m$ to show that for sufficiently large $m$ the bundle $\Sym^2\ke'_m \tensor \omega_S^2$ has no higher cohomology. 
Substituting into the formulas of Corollary \ref{cor: invariants} we see that 
 \begin{align*}
 K_{T_m}^2 & = 9m^2 A^2 + \text{lower order terms}\\
\chi(T_m) & = \frac{3}{2}m^2A^2 + \text{lower order terms}.
\end{align*}
So if we choose $m$ also large enough such that $K_{T_m}^2 >0$, then we conclude using Corollary \ref{cor: minimal general type}.  
\end{proof}

Choosing a more sophisticated, e.g. polynomial, scaling of the data by an ample divisor, one can find more asymptotics for the slope, but nothing too far away from $[5,6]$ -- this is not the region of surfaces of general type that attracts the most interest. The reason is that a certain balance has to be maintained in order to be admissible. We illustrate this with the following result.

\begin{prop}
  Consider the bundle $\ke'_m$, nearly split corresponding to $(\kl, C_1', C_2, C_3)$, where $C_1'\in |C_1 +mA|$ is sufficiently general. 
  Then for $m\gg 0 $ the bundle $\ke'_m$ is not admissible. 
\end{prop}
\begin{proof}
 We observe that $H^0(S, \kl(2C_2 + 2C_3 - (C_1+mA)\tensor \ki_{C_2\cap C_3})=0$ for $m\gg 0$, so the result follows from Proposition \ref{prop: necessary condition irreducibility}.
\end{proof}

\section{Special triple planes}\label{sect: triple planes}
 We consider in detail the case $S = \IP^2$, that is special triple planes. Let $\ke$ be an nearly split bundle associated to $(\kl, C_1, C_2, C_3)$ and let $d = \deg \kl$, $c_i = \deg C_i$ with $c_1\geq c_2\geq c_3\geq0$. 
 
 \begin{rem}\label{rem: split triple plane}
The above includes the case $c_3 = 0$, where the bundle splits as
\[ \ke \isom \ko_{\IP^2}(a_1) \oplus \ko_{\IP^2}(a_2)\]
with $a_i = d+c_i$. 

From Example \ref{exam: split triple planes}, we know that in this case the bundle is admissible if and only if it is gg-admissible if and only if 
\[ 0<a_2\leq a_1\leq 2 a_2.\]

Some things work slightly differently in this case, so we mostly  stick to the invariants $(a_1, a_2)$ in the split case instead of referring to it as the nearly split case with $d = c_3 = 0$. 
\end{rem}

 \begin{rem}
  By  \cite[Theorem~6.3.55]{LazarsfeldII} we see that on the projective plane an admissible 
$\ke$ is necessarily ample. The previous remark shows that the converse is not true, as $\ke = \ko(3)\oplus \ko(1)$ is ample but not admissible. Morally, the positivity of the bundle has to maintain a certain balance for admissibility.
 \end{rem}

The collected results from Section \ref{sect: suffcient and necessary conditions} give us a complete overview over the admissibility of $\ke$ also if all three curves are non-empty.  
 \begin{thm}\label{thm: admissible triple planes}
 Let $\ke$ be an nearly split bundle on $\IP^2$ as above with $c_3>0$.
 \begin{enumerate}
  \item The following are equivalent:
  \begin{enumerate}
   \item $\ke$ is trivially admissible;
   \item $\ke$ is gg-admissible;
   \item $d-c_1-c_2+2c_3\geq 0$.
  \end{enumerate}
\item The following are equivalent:
\begin{enumerate}
 \item $\ke$ is admissible 
 \item $\ke$ is gci-admissible;
  \item One of the following conditions holds:
  \begin{itemize}
   \item $d+c_3\geq c_1$, or 
   \item 
   $C_3$ is smooth at every point of $C_2 \cap C_3$, $c_1> d+c_3 \geq c_2$ and $d -c_1 +2c_2 - c_3 \geq 0$. 
   \end{itemize}
\end{enumerate}
 \end{enumerate}
 \end{thm}
 \begin{proof}
We start with $(i)$. The equivalence $(a)\iff (c)$ is the first item of Lemma \ref{lem: consequences of splitting} and $(a)\implies (b)$ has been proved in Proposition \ref{prop: implications admissibility}.   Assume that $(b)$ holds. 
  Then Lemma \ref{lem: nec for easy existence} implies that for $a+b = 3$ the line bundles 
    \[\det \ke ^\vee \tensor \kl^3(a(C_1 +C_2) + bC_3)|_{C_3} = \ko_{C_3}(d + (a-1)(c_1+c_2) + (b-1) c_3)  \]
  are globally generated, which for a plane curve just means that they have non-negative degree. In particular $(a,b)=(0,3)$ gives us condition $(c)$ and we have proved the equivalence.
  
Now we come to $(ii)$.   Since $H^1(\IP^2, \ko(d))=0$ for any $d$, the equivalence of $(a)\iff(b)$ follows from Lemma \ref{lem: consequences of splitting}. The equivalence $(b)\iff (c)$ comes from Theorem \ref{thm: gci existence}.
 \end{proof}

 For convenient reference we also specialise the formulas for the invariants to the case at hand.
 \begin{lem}\label{lem: invariants triple plane}
Assume $\ke$ is admissible with triple plane $T$. Then
\begin{align*}
   K_T^2
   & = 27-12(2d+\tsum_i c_i)+(5d^2+5\tsum_i dc_i+2\tsum_i c_i^2+\tsum_{i<j}c_ic_j), \\
   \chi(T) 
   &=3+d^2-3d+\tsum_i d c_i+\frac12 \tsum_i(c_i^2-3c_i), \\
   p_g(T) & =  - \binom {d-1}{2}+\sum_i \binom{d+c_i-1 }{2}\\
   q(T) &  = h^1(S, \ke\tensor \omega_S) =  1+ p_g(T) - \chi(T)
  \end{align*}
  In the split case  (Remark \ref{rem: split triple plane})   the same formulas hold with $d = c_3 = 0$ and $c_1 = a_1$, $c_2 = a_2$. 
 \end{lem}
\begin{proof}
 This is Corollary \ref{cor: invariants}, where the expression for $p_g$ is computed via the resolution  \eqref{eq: split resolution of E} and standard cohomology computations in the plane. 
\end{proof}

We now want to give an overview of the kinds of surfaces we can construct as special triple planes. Let us start with two simple examples.

\begin{exam}[The split case with $a_1 = a_2$]\label{exam: split a_1=a_2}
 If $\ke$ is split and $a = a_1 = a_2$, then we are considering hypersurfaces $T$ of bidegree $(a,3)$ in $\IP^2\times \IP^1$. By adjunction, the canonical divisor $K_T$ has bidegree $(a-3, 1)$ and is ample for $a\geq 4$. The general fibre of the projection $T\to \IP^1$ is a smooth  curve of genus $a$, so we get ruled surfaces for $a = 1,2$. For $a=3$ we get an elliptic fibration which coincides with the canonical map, thus a properly elliptic surface. 
\end{exam}

 \begin{exam}[The case $c_1 = c_2 = c_3$]\label{exam: all ci equal}
If   $c = c_1 = c_2 = c_3\geq 1$ and $d\geq 0$, then $\ke$ is trivially admissible by Theorem \ref{thm: admissible triple planes}. This case can easily be treated differently as follows: The projective bundle $\tilde \IP\isom \IP^2 \times \IP^2$ and the two defining equations \eqref{eq: eq for IP} and \eqref{eq: etatilde} for a triple cover $T$  are then bihomogenous of bidegree $(c,1)$ and $(d,3)$, that is, $T$ is a complete intersection in $\IP^2\times \IP^2$, from which its geometry is easily understood. Let $\pi'\colon \tilde \IP \to \IP^2$ be the projection to the second factor.

We discuss  some cases for a general choice of an equation.
\begin{description}
 \item[Cases $(d, c) = (0,1)$ and $(d, c) = (0,2)$] In the first case, the equation of degree $(0,3)$ defines  a product $\IP^2 \times E$ for an elliptic curve $E$. 
 
 Then an equation of bidegree $(1,1)$ cuts out a line in every fibre of the projection to $E$, so that $T\to E$ is a minimal ruled surface. 
 
 An equation of bidegree $(2, 1)$ instead makes $T\to E$ into a conic bundle, so $T$ is the fourfold blow up of a ruled surface over $E$. 
 \item[Case $(d,c) = (1,1)$] In this case $T$ is a complete intersection of bidegree $(1,1), (1,3)$, so the projection to the second factor actually has degree $1$, i.e., we have a birational map $\pi'\colon T \to \IP^2$. 
 Since any birational map is a composition of blow ups, a comparison of $K_T^2$ shows that $\pi'$ is the composition of $15$ blow ups. 
  \item[Case $(d,c) = (0,3)$] Then again the equation of bidegree $(0,3)$ shows that $T$ is contained in the a product $\IP^2\times E$ for an elliptic curve $E$. Since by adjunction
  $\omega_T \isom \pi'^*\ko_{\IP^2}(1)$, and $\pi'\colon T \to E$ is an elliptic fibration we see that $T$ is a minimal elliptic fibration of Kodaira dimension $1$.  
   \item[Cases $(d, c) = (1,2)$ and $(d, c) = (2,1)$]
   In both cases the map $\pi'\colon T \to \IP^2$ has degree $2$ and the canonical bundle is as in the previous case  $\omega_T \isom \pi'^*\ko_{\IP^2}(1)$. In particular $\omega_T$ is big and nef and $T$ is a minimal surface of general type. 
   
   These are Horikawa surfaces and, indeed, the canonical model of any minimal surface of general type with $\chi(T)=4$ and $K^2_T = 2$ is, via the canonical map, a double cover of the plane branched over an octic. 
   \item[The cases $d+c\geq 4$] In all the remaining cases we have $d+c\geq4$, and by adjunction we have that $\omega_T = \pi^*\ko_{\IP^2}(c+d-3)\tensor \pi'^*\ko_{\IP^2}(1)|_T$ is ample, so $T$ is minimal of general type. 
   \end{description}
   \end{exam}

We begin with a rough estimate to be refined later. 
\begin{lem}\label{lem: rough estimate for min gen type triple planes}
Assume $\ke$ is admissible. Then any  smooth triple cover $f\colon T \to \IP^2$ with Tschirnhausen bundle $\ke$ is a minimal surface of general type if one of the following holds:
\begin{enumerate}
 \item $c_1\geq 1$ and $d+c_3\geq 4$, 
 \item $\ke$ is split and $a_1\geq 4$.
\end{enumerate}
\end{lem}
\begin{proof}
 Under the first condition the bundle $\tilde \ke \tensor \omega_S(-1)$ is globally generated, hence Proposition \ref{prop: general criterion min gen type} applies. 
 
 In the other case, since all symmetric powers of a split bundle on $\IP^2$ have vanishing $H^1$, we apply the criterion of Corollary \ref{cor: minimal general type} combined with the formulas from Lemma \ref{lem: invariants triple plane}. We find then that $a_1\geq 4$ is both a necessary and sufficient condition. 
\end{proof}

\begin{thm}\label{prop: triple planes not general type}
Let $f\colon T \to S$ be a special triple plane with Tschirnhausen bundle $\ke$. If $T$ is not of general type, then the invariants are as follows: 
\begin{center}
  \begin{tabular}{cccc|ccc}
 \toprule
 $p_g$ & $q$ & $K_T^2$ & $\kappa$ & $d$ & $( c_1, c_2, c_3)$ & admissible\\
 \midrule
 $0$ & $0$ & $8$ & $-\infty$ & $0$ & $(1,1,0)$ & gg (split)\\
 $0$ & $0$ & $3$ & $-\infty$ & $0$ & $(2,1,0)$ & gg (split)\\
 $0$ & $0$ & $-1$ & $-\infty$ & $0$ & $(2,2,0)$& gg (split)\\
  $0$ & $1$ &  $0$ & $-\infty$ & $0$ & $(1,1,1)$ & gg\\
 $0$ & $1$ &  $-9$& $-\infty$ & $0$ & $(2,2,2)$ & gg\\
 $0$ & $0$ &  $-4$ &$-\infty$ & $1$ & $(1,1,1)$ & gg\\
 \midrule
 $1$ & $0$ & $-1$ & $0$ & $0$ & $(3,2,0)$ &gg (split)\\
  $1$ & $0$ &  $-3$ & $0$ & $1$ & $(2,1,1)$ & gg\\
\midrule 
 $2$ & $0$ & $0$ & $1$ & $0$ & $(3,3,0)$ & gg (split)\\
 $2$ & $0$ &  $-1$ & $1$ & $1$ & $(2,2,1)$ & \checkmark \\
 $3$ & $1$ &  $0$ & $1$ & $0$ & $(3,3,3)$ & gg\\
    \bottomrule
 \end{tabular}
\end{center}
\end{thm}
\begin{proof}
We start with the split case ($c_3=d = 0$). From Remark \ref{rem: split triple plane} and Lemma \ref{lem: rough estimate for min gen type triple planes} we have very few cases to consider, namely $1\leq a_2\leq a_1 \leq \min\{3, 2a_2\}$. The invariants, except the Kodaira dimension, are computed from Lemma \ref{lem: invariants triple plane} and the cases with $a_1 = a_2$ have been treated in Example \ref{exam: split a_1=a_2}.
To determine the Kodaira dimension in the remaining two cases we use the splitting of the bundle and Corollary \ref{cor: invariants} to show that  all $P_m(T)=0$ for $(a_1, a_2) = (2,1)$ and  all $P_m(T)=1$ for $(a_1, a_2) = (3,2)$. The latter case  is the resolution of the linear projection from an internal point of a quartic K3 surface in $\IP^3$, see Example \ref{exam: quartic}.

Now let us turn to the case where $c_3>0$. 
 An easy estimation using the formulas from Lemma \ref{lem: invariants triple plane} shows that if $d\geq 2$, then $K_T^2>0$ and $p_g>0$, so $T$ is of general type by Corollary \ref{cor: minimal general type}. Note that by Theorem  \ref{thm: admissible triple planes} the case $d=0$ can only occur if $c_1 = c_2 = c_3$ and these have been treated in Example \ref{exam: all ci equal}. 
 
So the remaining cases with $\kappa<2$ have  $d=1$, not all $c_i$ equal and,  by Lemma \ref{lem: rough estimate for min gen type triple planes},  $c_3\leq 2$, which by the inequalities in Theorem \ref{thm: admissible triple planes} allows for $  (c_1, c_2, c_3) $ to be one of
\begin{equation}\label{eq: cases d = 1}
 (2,1,1), (2,2,1), (3,2,1), (3,2,2),  (3,3,2), (4,2,1), (4,3,2), (5,3,2).
 \end{equation}
 Computing the invariants in each case we can exclude the cases where both $K_T^2>0$ and $p_g>0$, which leaves us with the two cases in the table. 
 
 A more detailed descriptions of the cases including a justification of the Kodaira dimension given can be found in Example \ref{exam: triple plane kappa = 1} and Section \ref{sect: K3 family }.
\end{proof}

The remaining special triple planes are all of general type. We conclude the general treatment by pointing out which of them are not minimal or give surfaces with reasonably small invariants.
\begin{prop}\label{prop: triple planes non min of gen type}
Let $f\colon T \to S$ be a special triple plane of general type with Tschirn\-hausen bundle $\ke$.
If $T$ is not minimal or $p_g(T)\leq 7$, then the invariants are as follows: 
\begin{center}
 \begin{tabular}{ccccc|ccc}
 \toprule
 $p_g$ & $q$ & $K_T^2$ & $\kappa$ & minimal  & $d$ & $( c_1, c_2, c_3)$ & admissible\\
 \midrule
   $3$ & $0$ &  $3$ & $2$& \checkmark & $0$ & $(4,2,0)$ & gg (split)\\
  $3$ & $0$ &  $2$ & $2$& \checkmark & $1$ & $(2,2,2)$ & gg\\
 $3$ & $0$ &  $2$ & $2$ & \checkmark & $2$ & $(1,1,1)$ & gg\\
 \midrule
    $4$ & $0$ &  $5$ & $2$& \checkmark & $0$ & $(4,3,0)$ & gg (split)\\
 $4$ & $0$ &  $5$ & $2$ & no  & $1$ & $(3,2,1)$ & \checkmark\\
 \midrule
 $5$ & $0$ &  $9$ & $2$& \checkmark & $1$ & $(3,2,2)$ & gg\\
 $5$ & $0$ &  $8$ & $2$& \checkmark & $2$ & $(2,1,1)$ & gg\\
 \midrule
 $6$ & $0$ &  $9$ & $2$& \checkmark & $0$ & $(4,4,0)$ & gg (split)\\
 \midrule
 $7$ & $0$ &  $14$ & $2$& \checkmark & $0$ & $(5,3,0)$ & gg (split)\\
 $7$ & $0$ &  $15$ & $2$& no & $1$ & $(4,2,1)$ & \checkmark\\
 $7$ & $0$ &  $17$ & $2$& \checkmark & $1$ & $(3,3,2)$ & \checkmark \\
 $7$ & $0$ &  $15$ & $2$ & \checkmark& $2$ & $(2,2,1)$ & gg\\
    \bottomrule
 \end{tabular}
\end{center}
\end{prop}
\begin{proof}
The split cases all satisfy the criterion given in Lemma \ref{lem: rough estimate for min gen type triple planes} and these are thus easily listed using Remark \ref{rem: split triple plane}.

For the cases with $c_3>0$ we want to use the criterion of Corollary \ref{cor: minimal general type} and determine, when $h^1(T, \omega_T^{\tensor2}) = h^1(\Sym ^2 \ke(-6)) = 0$. Partially applying Serre duality to the cohomology sequence of \eqref{eq: symmetric power of nearly split sequence} we get
 \[ 
\begin{tikzcd}
0\arrow{r} & H^1(\Sym^2\ke(-6))\arrow{r}&H^0(\tilde\ke^\vee(3-d))^\vee\arrow{r}&H^0(\Sym^2\tilde\ke^\vee(3))^\vee
\end{tikzcd}
\]
and by the splitting of $\tilde \ke$ the group in the middle is non-zero if and only if $3 - 2d - c_3\geq 0$, so $T$ is minimal of general type as soon as $d\geq 2$.

Again, the case where all $c_i$ are equal, which includes all cases with $d =0$,  was completely treated in Example \ref{exam: all ci equal}. 
It remains to look at the cases with $d = c_3 = 1$ and not all $c_i$ equal, which have already been listed in \eqref{eq: cases d = 1}. In both cases the sequence gives  $h^1(T, \omega_T^{\tensor2}) = 1$. 

Listing the remaining cases with $p_g\leq 7$ is an easy task once we note that Lemma \ref{lem: invariants triple plane} gives that $p_g\geq 8 $ if $d\geq 3$, so only the cases with $d=2$ have to be considered. 
\end{proof}
Some of the cases in the tables are discussed  in the examples below. 

\begin{exam}[The split case with $(a_1, a_2) = (3,2)$]\label{exam: quartic}
Let $u,v$ be the coordinates on the fibre and, as usual, $x,y,z$ the coordinates on the base of
\[ \pi\colon \IP = \IP(\ko_{\IP^2}(3) \oplus \ko_{\IP^2}(2)) \to S.\]
Let  $\km:= \ko_\IP(1) \tensor \pi^*\ko_{\IP^2}(-2)$. Then $\pi_*\km \isom \ko_{\IP^2}(1)\oplus \ko_{\IP^2}$, thus has a base of sections $xu$, $yu,$ $zu,$ $v$ and defines a morphism
$ \phi \colon \IP\to \IP^3$,
which is the blow-up at the coordinate point $ p = (0:0:0:1)$ and contracts the divisor $E = \{u=0\}$. Note that $\ko_{\IP}(E) = \ko_{\IP}(1)\tensor \pi^*\ko_{\IP^2}(-3)$. The triple cover $f\colon T\to \IP^2$ with Tschirnhausen bundle $\ko_{\IP^2}(3)\oplus\ko_{\IP^2}(2)$ is defined by a section of $ \ko_\IP(3)\tensor \pi^*\ko_{\IP^2}(-6) = \km^{\tensor 4} (-E)$. In other words, $\phi|_T\colon T \to \phi(T)= \bar T$ is the blow up of a quartic K3 surface $\bar T$ passing through $p$ and the triple cover is the resolution of the linear  projection $\bar T \dashrightarrow \IP^2$ from the internal point $p$.
\end{exam}

\begin{exam}[The case $d = 1$, $(c_1, c_2, c_3) = (2,2,1)$]\label{exam: triple plane kappa = 1}
Let us consider a triple cover $T \subset \tilde \IP$ defined by the equations \eqref{eq: eq for IP} and \eqref{eq: etatilde}, where $\tilde \eta$ does not contain the monomial $w^3$. In these coordinates we compute
\[ H^0(T, \omega_T) \isom H^0(\tilde \IP, \ko_{\tilde \IP}(1) \tensor \tilde \pi^*\omega_{\IP^2}) \isom H^0(\IP^2, \tilde \ke(-3)) = \langle u,v\rangle.\]

Let us compute the general element in the canonical pencil, $C_\lambda = \{ v - \lambda u=0\}\subset T$. Let $\gamma'_\lambda:=\gamma_1+\lambda\gamma_2$ and $C'_{\lambda}:=Z(\gamma'_{\lambda})$. We note first that the surface 
\[S_\lambda:=\{v-\lambda u=0, \ \gamma_1u+\gamma_2v+\gamma_3w=0\}\subset\tilde\IP\]
is precisely the blow up of $\IP^2$ at $Z_\lambda:=C_3\cap C'_{\lambda}$, defined via the natural map
\[ \begin{tikzcd}
 \ke  \rar[->>] &\kl (C'_{\lambda}+C_3)\ki_{Z_\lambda}
\end{tikzcd}\]
which then induces
\[ S_\lambda=\Proj_S \Sym^\bullet \ki_{Z_\lambda} \into \Proj_S\Sym^\bullet\ke(-4)= \IP.\]
Then we have $C_\lambda=S_\lambda\cap T$, given by the equations: 
\[ v-\lambda u=0, \ \gamma'_\lambda  u + \gamma_3 w = 0 \text{ and }   u(g_2 u^2 + g_1uw + g_0 w^2)=0,\]
with  $g_i\in H^0(\IP^2,\ko(i))$. Putting $B= \{u = v = \gamma_3 = 0\}\subset T$, which is sent by $f$ biholomorphically to $C_3$,  we have $C_\lambda = B + E_\lambda$, where $E_\lambda$ is defined by the equations 
\[ v - \lambda u =  \gamma'_\lambda  u + \gamma_3 w= g_2 u^2 + g_1uw + g_0 w^2=0.\]

Over the locus where $\gamma_3\neq 0$, these become:
\[v = \lambda u, \;  w = -\frac{\gamma'_\lambda}{\gamma_3} u \text{ and }  g_2  - \frac{\gamma'_\lambda g_1}{\gamma_3} + \frac{{\gamma'_\lambda}^2g_0}{\gamma_3^2} = 0.\]
In other words, $E_\lambda\subset S_\lambda$ is the strict transform of the plane quartic
 \[  Q_\lambda:=Z({\gamma_3^2} g_2  - {\gamma'_\lambda}{\gamma_3}g_1 + {\gamma'_\lambda}^2g_0)\subset S.\]
For generic $g_i$ and $\lambda$, $Q_\lambda$ has singularities precisely at $Z_\lambda$, which are two double points. Since moreover $E_\lambda$ intersects the exceptional divisor of $S_\lambda$
\[E=\{ \gamma'_\lambda=\gamma_3=0\}\subset S_\lambda\]
in two points over each $P\in Z_\lambda$, it follows that $E_\lambda$ is smooth and is the normalisation of $Q_\lambda$. Since $g(Q_\lambda)=3$, the normalisation sequence gives
$g(E_\lambda)=g(Q_\lambda)-2=1$.

Finally we see that for a general $\lambda$, $B\cap E_{\lambda}=\emptyset$, so that we get 
\[0=g(B)=1+\frac{1}{2}B(B+K_T)=1+B^2,\] 
therefore $B$ is a $-1$ one curve and $T$ is the blow-up at one point of a minimal elliptic surface. In particular, $\kappa(T)=1$.
\end{exam}

\begin{exam}[The case $d = 1$, $(c_1, c_2, c_3) = (3,2,1)$]
From Proposition \ref{prop: triple planes non min of gen type} we see that $T$ is of general type with  $K_T^2 = 5$ and $p_g(T) = 4$, but not minimal, and from its proof we see that it contains a unique $(-1)$-curve. Thus its minimal model $T_\text{min}$ satisfies $K_{T_\text{min}} ^2 = 6$ and $p_g(T_\text{min}) = 6$ and such surfaces were classified by Horikawa in \cite{Horikawa78}.

Let us analyse the canonical map of $T$. 
In the usual notation we have $H^0(\omega_T ) = \langle xu, yu, zu, v\rangle$ and the curve $B = \{u = v = \gamma_3 = 0\}\subset T \subset \tilde \IP$ is the fixed part of the linear system. So, as in the previous example, $B$ is the unique $(-1)$-curve. 

Over the complement $U$ of the line $C_3$, we can invert the equation $\gamma_3$ and use \eqref{eq: eq for IP} to eliminate $w$, so that $u,v$ are global homogeneous coordinates on the fibres of $\IP|_U$. Then we get a birational map over $\IP^2$
\[ 
\begin{tikzcd}
 \IP|_{U}\arrow{rr}{(xu: yu: zu: v)}\arrow {dr}[swap]{\pi} && \IP^3\arrow[dashed]{dl}{\text{projection from $(0:0:0:1)$}}\\
 & \IP^2
\end{tikzcd}
\]
where every fibre of $\pi$ over a point in $U$ maps isomorphically  to a line through $(0:0:0:1)$ in $\IP^3$. 
In particular, the canonical map of $T$ maps birationally onto a surface in $\IP^3$ -- it is of type $Ia$ according to Horikawa's notation. 

One can actually obtain an explicit equation for the image in the following way: substituting $w = - \frac{ \gamma_1 u + \gamma_2 v}{z}$ into the defining equation \eqref{eq: etatilde} 
results in an equation
\[ \sum _{a+b =3} \frac{h_{ab}}{z^2}u^a v^b, \qquad \deg (h_{ab}) = 3+a\]
defining the surface in $\IP_U$. Multiplying by $z^2 u^3$ this becomes an equation of degree six  in the generators of the canonical linear system. 

Thus the image is  the corresponding sextic surface, which has a triple point at $(0:0:0:1)$. 
\end{exam}

 \subsection{K3 surfaces of genus 4 from special triple planes}\label{sect: K3 family }

  Let us consider the nearly split bundle $\ke$ on $S = \IP^2$ defined by $d = 1$ and the three curves
 \[ C_1 = Z( z^2 + q'(x,y) ),\; C_2 = Z( x), \; C_3 = Z(y).\]
 Let $p = C_2 \cap C_3$ and $L$ a line in the plane so that the exact sequence \eqref{eq: Serre seq for E} becomes
 \begin{equation}\label{eq: E for K3 triple plane}
 0 \to \ko_S(L+C_1) \to \ke \to \ki_p(L+C_2+C_3) \to 0.
  \end{equation}

By Theorem \ref{prop: triple planes not general type} the bundle is trivially admissible and corresponding triple covers  $f\colon T \to S$ satisfy $p_g (T)= 1$, $q (T)= 0 $ and $K_T^2 = -3$; they will turn out to be blown-up K3 surfaces.

We start by studying a particular linear system on the projective bundle $\tilde\IP= \IP(\ke)$, where we consider the line bundle 
\[\tilde \km:=\ko_{\tilde\IP}(1)(-2\pi^*L).\]
Note that 
\[\km:=\tilde \km|_{\IP} = \ko_{\IP}(1)(-2\pi^*L), \ \   \km|_T = \omega_T(f^*L).\]

\begin{lem}\label{lem:  IP is blow up in K3 case}
The linear system $|\tilde\km|$ has no base points. The corresponding morphism $\tilde\phi: \tilde\IP\rightarrow \IP^4$ is the blow-up of a line $l\in\IP^4$ with the exceptional divisor $\IP(\ko_S(L+C_2)\oplus\ko_S(L+C_3))\subset\tilde\IP$. Moreover, there is a quadric $Q\subset\IP^4$ containing $l$, such that the morphism $\phi:\IP\rightarrow\IP^4$  corresponding to $|\km|$ is the blow-up of $Q$ at $l$ with the exceptional divisor
$$Y:=\Proj_S \Sym^\bullet \ki_p \into \IP$$
and  and fits into a diagram 
\[ 
 \begin{tikzcd}
  \IP \arrow{rr}{\phi}\arrow{dr}[swap]{\pi} && Q \rar[hookrightarrow] \arrow[dashed]{dl}& \IP^4 \arrow[dashed]{dll}{\text{projection from $l$}}\\
  & \IP^2
 \end{tikzcd}.
\]
\end{lem}
\begin{proof}
As before, we let $u$, $v$ and $w$ be the fiberwise coordinates on $\tilde\IP$ defined as in Remark~\ref{rem: IP and tilde IP}, and we note that we have the isomorphism
\[ 
 \begin{tikzcd}
 H^0(S,\ko_S(1))\oplus H^0(S,\ko_S)\oplus H^0(S,\ko_S)\arrow{r}{(u\tilde\pi^*,v,w)} & H^0(\tilde\IP,\tilde\km)
 \end{tikzcd}.
\]
Denoting by $\hat s=u\cdot\tilde\pi^*s$ for any $s\in H^0(S,\ko_S(1))$, it follows that $\hat x, \hat y, \hat z, v, w$ form a basis for $H^0(\tilde\IP,\tilde\km)$, and the corresponding rational morphism
 \[ \tilde\phi = (\hat x: \hat y : \hat z: v: w)\colon \tilde \IP \dashrightarrow \IP^4\]
 is  regular, since $\ko_S(1)$ is base point free and $\{u=v=w=0\}=\emptyset$. One notes in fact that $\tilde\IP$ is the blow-up of $\IP^4$ at $l:=\{x_0=x_1=x_2=0\}\subset\IP^4$, $\tilde\phi$ is the blow-up map and the exceptional divisor is given by 
 \[\tilde\IP\supset A:=\{u=0\}=\IP(\ko_S(L+C_2)\oplus \ko_S(L+C_3))\cong S\times \IP^1.\]
 Moreover, away from $l$, $\pi\circ\phi^{-1}$ is precisely the linear projection from $l$.

 Now, from the exact sequence
 \[ 0\rightarrow \ko_{\tilde\IP}(-\pi^*L)\rightarrow \tilde \km\rightarrow\km\rightarrow 0,\]
 one infers that the restriction induces an isomorphism $H^0(\tilde\IP,\tilde\km)\cong H^0(\IP,\km)$, so that the  rational map $\phi$ induced by $\km$ is precisely $\tilde\phi|_\IP$. Since $\IP\subset\tilde\IP$ is defined by the equation $\gamma_1u+\gamma_2v+\gamma_3w=0$, this gives a relation between the sections of $\km$ on $\IP$:
 \[\hat z^2+q'(\hat x,\hat y)+\hat x v+\hat y w=0\]
which implies that the image of $\phi=\tilde\phi|_\IP$ is contained in $Q:=Z(q)$ for 
$$q:=x_2^2-q'(x_0,x_1)+x_0x_3+x_1x_4.$$ 
In fact, since $\phi$ is generically a submersion and $Q$ is irreducible and $3$-dimensional, it follows that the image of $\phi$ is precisely $Q$. Finally,  as $\IP$ is irreducible, we conclude that $\IP$ is the strict transform of $Q$ under $\tilde\phi$, and since $Q$ contains $l$,  $\IP$ must be the blow-up of $Q$ at $l$, with the exceptional divisor
\[A\cap\IP=\{u=0, \ \gamma_1u+\gamma_2v+\gamma_3w=0\}=Y.\qedhere\] 
\end{proof}

 \begin{lem}\label{lem: comparing linear systems in K3 case}
  The map $\phi$ induces a surjection and an isomorphism
  \[ 
  \begin{tikzcd}
H^0(\IP^4, \ko_{\IP^4}(3))\rar[->>] & H^0(Q, \ko_Q(3)) \rar{\phi^*}[swap]{\isom}& H^0(\IP, \ko_{\IP}(3) \tensor \pi^*\det \ke^\vee).   
  \end{tikzcd}
\]
In particular, every triple cover $f\colon T \to S$ with Tschirnhausen bundle $\ke$ is the pullback via $\phi$ of a cubic section of $Q$ not containing the line $l$. 

  \end{lem}
\begin{proof}
 The first part follows from the diagram
 \begin{small}\[ 
  \begin{tikzcd}
H^0(\IP^4,\ko_{\IP^4}(1)) \dar[hookrightarrow]\rar{\isom} &H^0(\tilde\IP,\ko_{\tilde\IP}(-5\tilde\pi^*L)) \rar{\isom}\dar[hookrightarrow]& H^0(S,\Sym^2\tilde \ke (-5L)   ) \dar[hookrightarrow]\\
   H^0(\IP^4,\ko_{\IP^4}(3)) \dar[->>] \rar{\isom}[swap]{\tilde\phi^*} & H^0(\tilde\IP,3\tilde\km) \dar \rar{\isom} & H^0(S,\Sym^3 \tilde \ke( -6L))\dar[->>] \\
   H^0(Q, \ko_Q(3)) \rar{\phi^*} & H^0(\IP,3\km)  \rar{\isom} & H^0(\Sym^3\ke(-6L)) 
  \end{tikzcd}, \]
  \end{small}
where the last column comes from the  appropriate twist of \eqref{eq: symmetric power of nearly split sequence}.

Now the group on the right hand side contains the sections defining triple covers $T$ with Tschirnhausen bundle $\ke$. If the corresponding cubic section contained the line $l$, then its pullback $T$ would contain the divisor $Y$ and with that the full fibre over $p$, which cannot happen. 
\end{proof}

 \begin{prop}\label{prop: K3 construction}
  Let $f\colon T\to \IP^2$ be a special triple plane with Tschirnhausen bundle $\ke$ as above. Then there is a commutative diagram
  \[
   \begin{tikzcd}
{} & T_0 \arrow[hookrightarrow]{r} & Q \arrow[hookrightarrow]{r} & \IP^4\arrow[dashed]{ddll}{\pi_l}   \\
T \arrow{ur}{\phi_T}
\rar[hookrightarrow]\arrow{dr}[swap]{f} & \IP(\ke)\dar{\pi} \arrow{ur}{\phi}
\\
& \IP^2
   \end{tikzcd}
  \]
where
\begin{enumerate}
\item $Q$ is a smooth quadric threefold;
 \item $\pi_l$ is the projection  from a line $l\subset Q$;
 \item $T_0= Q\cap R$ is the complete intersection of $Q$ with a cubic hypersurface $R$ not containing $l$;
 \item $\phi$ is the blow up of $l$ in $Q$;
 \item $\phi_T$ is the blow up at $l\cap R = l \cap T_0$, i.e., at three points if $R$ is sufficiently general.
\end{enumerate}
If $T$, respectively $R$, is sufficiently general, then $T_0$ is a smooth K3 surface with a polarisation of degree $6$, and $T$ is its blow up in three points. 
\end{prop}
\begin{proof}
 All statements follow immediately from Lemma \ref{lem:  IP is blow up in K3 case}, Lemma \ref{lem: comparing linear systems in K3 case} and the properties of the blow up once we note that a general  complete intersection of a quadric and a cubic in $\IP^4$ is indeed a K3 surface with a very ample polarisation of degree $6$. 
\end{proof}

 \begin{rem}
Let $T_0$ be a K3 surface containing a very ample line divisor $D$ with $D^2 = 6$. Possibly changing $D$ in the linear system, we may assume that $D$ is a smooth curve of genus $4$. 

It is well known that $|D|$ embeds $T_0 = Q\cap R$ as the complete intersection of a quadric and a cubic in $\IP^4$ and that $D$ is embedded into a hyperplane $\IP^3\subset \IP^4$ via its canonical embedding. 

We have seen above that  a blow up of $T_0$ is a special triple plane if and only if the quadric $Q$ is smooth. 
By \cite[p.~206]{ACGH} this depends on the geometry of $D$ in the following way: recall that a $g_3^1$ on $D$ is a linear system of degree $3$ with at least two sections. Since $D$ is canonically embedded and thus not hyperelliptic, any $g_3^1$ on $D$ is base point free and makes $D$ into a trigonal curve. In the case of a curve of genus $4$ each $g_3^1$ comes from he projection from a line  the quadric $D \subset Q \cap \IP^3$. So either we have two distinct $g^1_3$'s if $Q\cap \IP^3$ is smooth or we have a unique $g^1_3$ if $Q\cap \IP^3$ is a quadric cone.

Summarising we have found that $T_0$ arises as in Proposition \ref{prop: K3 construction} if and only if the general element  $|D|$ carries two distinct $g^1_3$'s and the map $\phi_T \colon T \to T_0$ blows up one of the $g^1_3$'s, considered as points on $T_0$. 
 \end{rem}

 \begin{rem}
  If $Q$ is singular, then the projection from a line $l\subset Q\subset \IP^4$ still defines, after blow up, a triple cover
  \[ 
   \begin{tikzcd}
    T \rar{\phi_T} \arrow{dr}[swap]{f} & T_0\dar[dashed]\\
     & \IP^2
   \end{tikzcd}
.
  \]
  However, in this case the Tschirnhausen bundle is not nearly split.
 \end{rem}


 \end{document}